\documentclass[british,11pt,a4paper]{amsart}

\IfFileExists{tms.sty}{\usepackage[print]{tms}}{
\usepackage[dvipsnames]{xcolor}
\definecolor{url}{RGB}{0,113,185} 
\definecolor{link}{RGB}{160,0,0} 
\PassOptionsToPackage{ocgcolorlinks=true,citecolor=link,linkcolor=link,urlcolor=url,backref=page}{hyperref}
\PassOptionsToPackage{paperwidth=460pt,paperheight=700pt,height=590pt,width=360pt}{geometry}
}
\usepackage{mlmodern,dsfont}
\let\mathbb\mathds

\usepackage{microtype,amssymb,mathtools,hyperref}
\usepackage[T1]{fontenc}
\usepackage[marginratio=1:1]{geometry}
\usepackage[capitalize]{cleveref}
\usepackage[dvipsnames]{xcolor}

\numberwithin{equation}{section}

\renewcommand*{\d}{\mathrm{d}}
\newcommand*{\R}{\mathbb{R}}
\newcommand*{\tr}{\operatorname{tr}}
\newcommand*{\Ad}{\operatorname{Ad}}
\newcommand*{\ad}{\operatorname{ad}}
\newcommand*{\End}{\operatorname{End}}

\newtheorem{thm}{Theorem}[section]
\newtheorem{lem}[thm]{Lemma}
\newtheorem{prop}[thm]{Proposition}

\theoremstyle{definition}

\theoremstyle{remark}
\newtheorem{rmk}[thm]{Remark}

\title[Einstein metrics on cotangent bundles]{Induced almost para-K\"ahler Einstein\\ metrics on cotangent bundles}
\date{22nd July 2024}
\author[A.~{\v{C}}ap]{Andreas {\v{C}}ap}
\address{A.{\v{C}}. -- Faculty of Mathematics, University of Vienna, Vienna, Austria}
\email{Andreas.Cap@univie.ac.at}
\author[T.~Mettler]{Thomas Mettler}
\address{T.M. -- Faculty of Mathematics and Computer Science, UniDistance Suisse, Brig, Switzerland}
\email{thomas.mettler@fernuni.ch, mettler@math.ch}

\begin{document}

\begin{abstract}
In earlier work we have shown that for certain geometric structures on a smooth manifold $M$ of dimension $n$, one obtains an almost para-K\"ahler--Einstein metric on a manifold $A$ of dimension $2n$ associated to the structure on $M$. The geometry also associates a diffeomorphism between $A$ and $T^*M$ to any torsion-free connection compatible with the geometric structure. Hence we can use this construction to associate to each compatible connection an almost para-K\"ahler--Einstein metric on $T^*M$. In this short article, we discuss the relation of these metrics to Patterson--Walker metrics and derive explicit formulae for them in the cases of projective, conformal and Grassmannian structures.
\end{abstract}

\maketitle

\section{Introduction}

An \emph{almost para-K\"ahler structure} on a $2n$-manifold $N$ consists of a
pseudo-Rie\-mann\-ian metric $h$ of split-signature $(n,n)$ and a symplectic form
$\Omega$ such that the endomorphism $I : TN \to TN$ defined by
$\Omega=h(\cdot,I\cdot)$ satisfies $I^2=\mathrm{Id}_{TN}$. The
  point-wise eigenspaces of $I$ fit together to define smooth rank $n$
  subbundles of $TN$ which are isotropic for $h$ and Lagrangian for
  $\Omega$. Therefore, almost para-K\"ahler structures are sometimes also called
\emph{almost bi-Lagrangian structures}, see for
instance~\cite{MR1824987,MR2193747,MR3916049}. A torsion-free affine
  connection $\nabla$ on a smooth $n$-manifold $M$ canonically determines a
  split-signature metric on the total space of the cotangent bundle $T^*M$ of $M$,
  called the \emph{Patterson-Walker metric} \cite{MR48131}. For this metric, the vertical and
  horizontal subbundle of $TT^*M$ are isotropic.

  Several modified versions of the Patterson-Walker metric have been constructed,
  often motivated by compatibility conditions with projective equivalence of
  connections. For example, in \cite{MR4095182}, the authors use a Fefferman-type
  construction to obtain a modification such that projectively equivalent connections
  lead to conformally related metrics. Based on tools from projective differential
  geometry, a different kind of modification has been obtained in
  \cite{MR3833818}. The advantage of the latter modification is that it always
  produces an Einstein metric on $T^*M$, while Einstein metrics almost never occur in
  the construction of \cite{MR4095182}.

  Our general approach recovers these results from \cite{MR3833818}. For a
  torsion-free linear connection $\nabla$ on $TM$, we denote by $\boldsymbol{\mathsf{P}}^p$ the
  (projective) Schouten tensor, see \cref{sec:projective} for the precise
  definition. Denoting by $\nu:T^*M\to M$ the bundle projection and by
  $\tau\in\Omega^1(T^*M)$ the tautological one-form, we prove:
  \begin{prop}\label{prop:projective}
  Let $h_0$ be the Patterson--Walker metric of $\nabla$. Then
  $$h_0-\nu^*(\operatorname{Sym}\boldsymbol{\mathsf{P}}^p)-\tau\otimes\tau$$ is an Einstein metric on
  $T^*M$, which together with $\Omega=-\d\tau+\nu^*(\operatorname{Alt}\boldsymbol{\mathsf{P}}^p)$ defines
  an almost para-K\"ahler structure. Projectively equivalent connections lead to
  isometric metrics.
  \end{prop}
  Moreover, we construct different Einstein modifications of the Patterson Walker
  metric for connections that are compatible with certain geometric structures, the
  best known example of which are conformal structures. So suppose that $g$ is a
  pseudo-Riemannian metric on $M$ and let $\nabla$ be a Weyl connection for $g$ so that $\nabla g=\beta\otimes g$ for some
  $\beta\in\Omega^1(M)$. Then we let $\boldsymbol{\mathsf{P}}^c$ be the conformal Schouten tensor of
  $\nabla$, see \cref{sec:conformal} for the precise definition. The inverse
  metric $g^\#$ to $g$ defines a smooth function $|\cdot|^2_{g^\#}$ on $T^*M$ and we
  prove:
  \begin{prop}\label{prop:conformal}
  Let $h_0$ be the Patterson--Walker metric of a Weyl connection $\nabla$ for $g$. Then
  $$
  h_0-\nu^*(\operatorname{Sym}\boldsymbol{\mathsf{P}}^c)-\tau\otimes\tau+\tfrac12|\cdot|^2_{g^\#}\nu^* g
  $$ is an Einstein metric on $T^*M$, which together with
  $\Omega=-\d\tau+\nu^*(\operatorname{Alt}\boldsymbol{\mathsf{P}}^c)$ defines an almost para-K\"ahler
  structure. Different Weyl connections for $g$ lead to isometric metrics.
  \end{prop}
  There is also a version for torsion-free connections compatible with a Grassmannian
  structure, see \cref{sec:Grassmannian}, which requires a more subtle
  modification of the tautological one-form $\tau$.

 Conformal and (almost) Grassmannian structures are examples of a family
  of very specific G-structures, for which the structure group $G_0$ is obtained from
  a so-called $|1|$-grading of a simple Lie algebra. A fundamental result on these
  structures is that they admit an equivalent description via Cartan geometries whose
  homogeneous model is a generalized flag manifold. On the level of Cartan
  geometries, projective structures fit into this framework as well. In this case, the
  underlying $G_0$-structure is the full frame bundle and the extension to a Cartan
  geometry is equivalent to the choice of a projective equivalence class of
  connections.

In~\cite{MR4604420}, we showed how to canonically associate to a torsion-free
geometry of the above types on an $n$-manifold $M$ an almost para-K\"ahler-Einstein structure
$(h_A,\Omega_A)$ on the total space of a natural affine bundle $\mu:A\to M$ such that
$\dim(A)=2n$. The modelling vector bundle for $A$ is the
  cotangent bundle $\nu : T^*M \to M$ and sections of $\mu:A\to M$ can be interpreted
  as the \emph{Weyl structures} of the parabolic geometry. Any Weyl structure comes
  with a \emph{Weyl connection}, a principal connection on the $G_0$-structure, which
  is uniquely determined by the induced linear connection $\nabla$ on $TM$. For
  projective structures, the Weyl connections are exactly the connections in the
  projective class; in the other cases they are exactly the torsion-free connections
  compatible with the $G_0$-structure. In particular, in the conformal case Weyl connections agree with the usual notion of a Weyl connection as mentioned above. A Weyl connection also determines a \emph{Rho-tensor}
  $\boldsymbol{\mathsf{P}}$, a $\binom02$-tensor field on $M$, which is an analogue of the Ricci
  curvature, and specializes to the Schouten tensors for projective and conformal
  structures.  In \cite{MR4604420} it is also shown that the choice of a Weyl
structure induces a diffeomorphism $\varphi: T^*M \to A$ satisfying
\begin{equation}\label{eqn:Omega}
\varphi^*\Omega_A=-\d\tau+\nu^*(\operatorname{Alt}\boldsymbol{\mathsf{P}}),
\end{equation}
where $\tau$ as before denotes the tautological $1$-form of $T^*M$. 

The purpose of this note is to relate the metric $\varphi^*h_A$ to the Patterson--Walker metric of the Weyl connection $\nabla$. In particular, we show -- see \cref{thm:main}
below -- that $\varphi^*h_A$ can be written in a uniform way as
\begin{equation}\label{eqn:universalstructure} 
\varphi^*h_A=h_0-\nu^*(\operatorname{Sym}\boldsymbol{\mathsf{P}})+\boldsymbol{q},
\end{equation}
where $h_0$ is the Patterson--Walker metric for $\nabla$. Further, $\boldsymbol{q}$
is a symmetric covariant $2$-tensor field on $T^*M$ which only depends on the
underlying geometric structure and which is semibasic for the projection $\nu : T^*M
\to M$. Moreover, $\boldsymbol{q}$ is homogeneous of degree $2$ in the fibres of
$T^*M$, that is, satisfies $(\mathcal{S}_{t})^*\boldsymbol{q}=t^2\boldsymbol{q}$
where $\mathcal{S}_t : T^*M \to T^*M$ denotes scaling of a cotangent vector by the
factor $t \in \R\setminus\{0\}$. Hence for any Weyl connection $\nabla$,
  \eqref{eqn:universalstructure} defines an Einstein metric on $T^*M$ and
  \eqref{eqn:Omega} makes this into a para-K\"ahler structure.

  From this general result, \cref{prop:projective}, \cref{prop:conformal} and
  their analogue for Grassmannian structures follow from explicitly describing the
  ingredients in \eqref{eqn:universalstructure} for the structures in question. This
  is done in \cref{sec:examples}.

\subsection*{Acknowledgements} A.\v C.\ acknowledges support by the Austrian
Science Fund (FWF): 10.55776 / P33559. T.M. is partially supported by the DFG priority
programme ``Geometry at infinity'' SPP 2026. This article is based upon
  work from COST Action CaLISTA CA21109 supported by COST (European Cooperation in
  Science and Technology). https://www.cost.eu. The authors are grateful to M.~Dunajski for his helpful correspondence and to the anonymous referee for a careful report and valuable remarks.

\section{Preliminaries}

We start by collecting some basic facts about soldering forms, Patterson--Walker metrics, $|1|$-graded parabolic geometries, Weyl structures and the Rho tensor. Throughout this article all manifolds and mappings are assumed to be smooth, that is, $C^{\infty}$.

\subsection{Soldering forms}\label{sect:soldering}

Let $M$ be an $n$-manifold and $V$ a real $n$-dimensional vector space. We consider a
right principal $G_0$-bundle $\pi_0 : \mathcal{G}_0 \to M$ for some Lie group $G_0$
and a representation $\rho : G_0 \to \mathrm{GL}(V)$ of $G_0$ on the vector space
$V$. Suppose that we have a $1$-form $\omega \in \Omega^1(\mathcal{G}_0,V)$ on
$\mathcal{G}_0$ with values in $V$ which is semibasic and $\rho$-equivariant. The
former condition means that $\omega$ vanishes on all vectors of $T\mathcal{G}_0$ that
are tangent to the fibres of $\pi_0$ and the latter condition means that $\omega$
satisfies $R_a^*\omega=\rho(a^{-1})(\omega)$ for all $a \in G_0$, where $R_a :
\mathcal{G}_0 \to \mathcal{G}_0$ denotes the right action by $a$. Recall that such a
form defines a $1$-form $\boldsymbol{\omega} \in \Omega^1(M,E)$ on $M$ with values in
the vector bundle $\nu : \mathcal{G}_0 \times_{\rho}V=:E \to M$ associated to
$\rho$. Indeed, for all $v \in TM$ the element $\boldsymbol{\omega}(v) \in E$ is
represented by $(u,\omega(\tilde{v})) \in \mathcal{G}_0\times V$ for any choice of $u
\in \mathcal{G}_0$ having the same basepoint as $v$ and any $\tilde{v} \in
T_u\mathcal{G}_0$ such that $(\pi_0)^{\prime}_u(\tilde{v})=v$, where $(\pi_0)^{\prime}_u :
T_u \mathcal{G}_0 \to T_{\pi(u)}M$ denotes the derivative of $\pi_0$ at $u$. The
$1$-form $\omega \in \Omega^1(\mathcal{G}_0,V)$ is called a \emph{soldering form} if
$\boldsymbol{\omega}$ -- thought of as a map $TM \to E$ -- is a vector bundle
isomorphism. Recall that on the one hand this is equivalent to the fact that $\omega$
is \textit{strictly horizontal} in the sense that its kernel in any point of
$\mathcal{G}_0$ is the vertical subbundle. On the other hand,
  $\boldsymbol{\omega}$ induces a homomorphism from $\mathcal{G}_0$ to the linear
  frame bundle of $M$, whence this bundle defines a (first order) $G_0$-structure on
  $M$.

\subsection{The Patterson--Walker metric}\label{sec:Patterson}

We review the construction of the Patterson--Walker metric, adapted to our setting. To this end suppose, as above, that $\pi : \mathcal{G}_0 \to M$ is a right principal $G_0$-bundle equipped with a soldering form $\omega \in \Omega^1(\mathcal{G}_0,V)$ which is equivariant with respect to a representation $\rho : G_0 \to \mathrm{GL}(V)$. Assume in addition that $\rho$ is infinitesimally effective, that is, the induced Lie algebra representation $\varrho : \mathfrak{g}_0 \to \mathfrak{gl}(V)$ is injective. Using $\omega$, the tangent bundle of $M$ can be identified with the bundle associated to $\rho$, that is, $TM\simeq \mathcal{G}_0\times_{\rho} V$. 

Now let $\vartheta \in \Omega^1(\mathcal{G}_0,\mathfrak{g}_0)$ be a principal $G_0$-connection on $\mathcal{G}_0$. Let us denote by $V^*$ the dual space to $V$ and by $\varrho^* : \mathfrak{g}_0 \to \mathfrak{gl}(V^*)$ the the dual representation to $\varrho$.  Then on the product $\mathcal{G}_0 \times V^*$ we consider the $V^*$-valued $1$-form 
\[
\zeta_{\vartheta}=\d \xi+(\varrho^* \circ \vartheta)(\xi),
\] 
where $\xi : \mathcal{G}_0 \times V^* \to V^*$ denotes the projection onto the second factor. By construction, the $V^*$-valued $1$-form is equivariant with respect to the dual representation $\rho^* : G_0 \to \mathrm{GL}(V^*)$ and it is easy to check that it is semibasic for the projection $\mathcal{G}_0 \times V^* \to T^*M\simeq \mathcal{G}_0\times_{\rho^*} V^*$. Consequently, it represents a $1$-form $\boldsymbol{\zeta}_{\vartheta} \in \Omega^1(T^*M,\nu^*T^*M)$ on $T^*M$ with values in the pullback of the cotangent bundle $\nu : T^*M \to M$ of $M$. The pullback bundle $\nu^*T^*M$ is naturally a subbundle of $T^*(T^*M)$ and hence we may interpret $\boldsymbol{\zeta}_{\vartheta}$ as a $1$-form on $T^*M$ with values in $T^*(T^*M)$, or equivalently, as a section of $T^*(T^*M)\otimes T^*(T^*M)$. The symmetric part $h_{\vartheta}=\operatorname{Sym}\boldsymbol{\zeta}_{\vartheta}$ is a pseudo-Riemannian metric of split-signature $(n,n)$ known as the \emph{Patterson--Walker metric} associated to the connection $\vartheta$, see \cite{MR48131} for the original construction and \cite{MR3830122,MR4095182} for recent applications.

\subsection{Parabolic geometries}\label{sec:parabol}

Recall that a Cartan geometry of type $(G,P)$ for some Lie group $G$ and closed subgroup $P\subset G$ is a pair $(\pi : \mathcal{G} \to M,\theta)$ consisting of a right principal $P$-bundle $\pi : \mathcal{G} \to M$ together with a \emph{Cartan connection} $\theta$ taking values in the Lie algebra $\mathfrak{g}$ of $G$. We let $R : \mathcal{G}\times P \to \mathcal{G}$ denote the right action of $P$ and we write $R_g=R(\cdot,g)$ for all $g \in P$ and $\iota_u=R(u,\cdot)$ for all $u \in \mathcal{G}$. The $1$-form $\theta \in \Omega^1(\mathcal{G},\mathfrak{g})$ being a \emph{Cartan connection} means that for all $u \in \mathcal{G}$ the linear map $\theta_u : T_u\mathcal{G} \to \mathfrak{g}$ is an isomorphism and moreover
\begin{equation}\label{eqn:pullbackrightaction}
(R^*\theta)_{(u,g)}=(\iota_u^*\theta)_g+(R_g^*\theta)_u=(\Upsilon_P)_g+\operatorname{Ad}(g^{-1})\circ\theta_u,
\end{equation}
for all $(u,g) \in \mathcal{G}\times P$, where $\Upsilon_P$ denotes the Maurer--Cartan form of $P$ and $\Ad : G \to \mathrm{GL}(\mathfrak{g})$ the adjoint action of $G$. In addition, the curvature $2$-form $\Theta=\d \theta+\frac{1}{2}[\theta,\theta]$ of $\theta$ satisfies a further condition called \emph{normality}. The details of these conditions are not important for our purposes; however some consequences will be discussed below. We will always assume that the Cartan geometry is \emph{torsion-free}, that is, $\Theta$ has values in $\mathfrak{p}\subset\mathfrak{g}$, the Lie algebra of $P$. 

Here we consider the special case of $|1|$-graded parabolic geometries. This means that $G$ is assumed to be semisimple and that its Lie algebra $\mathfrak{g}$ is endowed with a so-called $|1|$-grading. This is a decomposition
\[
\mathfrak{g}=\mathfrak{g}_{-1}\oplus \mathfrak{g}_0\oplus \mathfrak{g}_1
\]
into a direct sum of linear subspaces $\mathfrak{g}_{i}$ for $i=-1,0,1$ such that $[\mathfrak{g}_i,\mathfrak{g}_j]\subset \mathfrak{g}_{i+j}$ with the convention that $\mathfrak{g}_{\ell}=\{0\}$ for $\ell =\pm 2$. Furthermore, no simple ideal of $\mathfrak{g}$ is allowed to be contained in $\mathfrak{g}_0$ and the Lie algebra $\mathfrak{p}$ of $P$ satisfies $\mathfrak{p}=\mathfrak{g}_0\oplus \mathfrak{g}_1$. In particular, this implies that $P\subset G$ is a parabolic subgroup in the sense of representation theory.

These properties have some nice consequences. First, the exponential map of $\mathfrak{g}$ -- restricted to $\mathfrak{g}_1$ -- is a diffeomorphism from $\mathfrak{g}_1$ onto a closed normal subgroup $P_+\subset P$. Second, defining $G_0\subset P$ to consist of those elements $g$ so that the adjoint action $\Ad(g) \in \mathrm{GL}(\mathfrak{g})$ preserves the grading of $\mathfrak{g}$, one can show that the Lie algebra of $G_0$ is $\mathfrak{g}_0$ and that $G_0$ is isomorphic to the quotient $P/P_+$. Third, every element $g$ of $P$ can be written as $g=g_0\exp(Z)$ for unique elements $g_0 \in G_0$ and $Z \in \mathfrak{g}_1$.

Since $P_+\subset P$ is a normal subgroup, we obtain a principal $P/P_{+}\simeq G_0$-bundle $\mathcal{G}/P_+ \to M$ whose total space we denote by $\mathcal{G}_0$ and whose basepoint projection we denote by $\pi_0$. We can project the values of the Cartan connection $\theta$ to $\mathfrak g/\mathfrak p\cong\mathfrak g_{-1}$. Equivariance of $\theta$ then easily implies that the result descends to a 1-form $\omega\in\Omega^1(\mathcal{G}_0,\mathfrak g_{-1})$. Equivariance of $\theta$ also implies that $\omega$ is equivariant with respect to the $G_0$-representation $\rho : G_0 \to \mathrm{GL}(\mathfrak{g}_{-1})$ obtained by restricting the adjoint representation to $\mathfrak{g}_{-1}$;
\begin{equation}\label{eqn:adjoint}
\rho=\Ad(\,\cdot\,)|_{\mathfrak{g}_{-1}}: G_0 \to \mathrm{GL}(\mathfrak{g}_{-1}), \qquad g_0 \mapsto \Ad(g_0)|_{\mathfrak{g}_{-1}}.
\end{equation}
This representation is infinitesimally effective. The construction then easily implies that $\omega$ is a soldering form in the sense of \cref{sect:soldering}. Hence we obtain a $G_0$-structure on the manifold $M$ and it turns out that, apart from the case of projective structures, the normality condition on $\Theta$ ensures that the Cartan geometry $(\mathcal{G}\to M,\theta)$ is an equivalent encoding of this $G_0$-structure. 

\subsection{Weyl structures}\label{subsec:weylstructures}

In order to work explicitly with a parabolic geometry it is often advantageous to fix a Weyl structure for the parabolic geometry. This gives a description of the Cartan geometry $(\mathcal{G}\to M,\theta)$ in terms of the underlying $G_0$-structure $(\mathcal{G}_0\to M,\omega)$ defined above. Here we briefly review the key facts and refer the reader to \cite{MR4604420,MR1997873,MR2532439} for details and additional context.

Following \cite{MR1997873}, we define a Weyl structure for $(\pi : \mathcal{G} \to M,\theta)$ to be a $G_0$-equivariant section $\sigma : \mathcal{G}_0 \to \mathcal{G}$ of the projection $\mathcal{G}\to\mathcal{G}_0$. Writing the Cartan connection $\theta$ as $\theta=(\theta_{-1},\theta_0,\theta_1)$ with $\theta_i$ taking values in $\mathfrak{g}_i$, a choice of Weyl structure $\sigma$ gives three $1$-forms on $\mathcal{G}_0$
\begin{equation}\label{eqn:pullbackalongweylstructure}
\begin{aligned}
\omega&=\sigma^*\theta_{-1} \in \Omega^1(\mathcal{G}_0,\mathfrak{g}_{-1}),\\
\vartheta_{\sigma}&=\sigma^*\theta_0 \in \Omega^1(\mathcal{G}_0,\mathfrak{g}_0),\\
\mathsf{P}^{\sigma}&=-\psi_B\circ (\sigma^*\theta_1) \in \Omega^1(\mathcal{G}_0,\mathfrak{g}_{-1}^*).
\end{aligned}
\end{equation}
Here $\omega$ is just the soldering form from \cref{sec:parabol} above. The form $\vartheta_{\sigma}$ is a principal $G_0$-connection on $\pi_0 : \mathcal{G}_0 \to M$ referred to as the \emph{Weyl connection} determined by $\sigma$. For the last component, we let $B$ denote a suitable constant multiple of the Killing form of $\mathfrak{g}$ and $\psi_B : \mathfrak{g}_{1} \to \mathfrak{g}_{-1}^*$ the linear map defined by the rule
\[
\psi_B(Z)(X)=B(Z,X)
\]
for all $Z \in \mathfrak{g}_1$ and all $X \in \mathfrak{g}_{-1}$. It is well-known that $\psi_B$ is an isomorphism. In the literature on parabolic geometries, $\psi_B$ is usually suppressed from the notation and $\mathfrak g_1$ is simply identified with $(\mathfrak g_{-1})^*$.

The multiple $B$ of the Killing form is chosen so that the form $\mathsf{P}^{\sigma}$ represents the so-called \emph{Rho tensor} $\boldsymbol{\mathsf{P}}^{\sigma}$ of the Weyl structure $\sigma$, thought of as a $1$-form on $M$ with values in $T^*M$. Notice that here our convention for the Rho tensor agrees with the classical definitions for conformal and projective structures and hence differs by a sign from \cite{MR4604420,MR1997873,MR2532439}.  

  \subsection{On the Rho tensor}\label{sec:Schouten}
  Here we briefly explain how the normalization condition on the curvature of a $|1|$-graded parabolic geometry leads to a way to explicitly determine the Rho tensor associated to a Weyl structure. We start by introducing a canonical object on $M$, which will also be useful for other purposes. Observe that we have a $G_0$-invariant multilinear map
  \[
\mathfrak{g}_{-1}\times \mathfrak{g}_1 \times \mathfrak{g}_{-1}\times \mathfrak{g}_1 \to \R, \qquad (Z_1,W_1,Z_2,W_2) \mapsto B([[Z_1,W_1],Z_2],W_2) 
  \] 
  which induces a $\binom{2}{2}$-tensor field $\boldsymbol{\mathsf{T}}$ on $M$. We will interpret this below as associating to a vector field $X\in\mathfrak X(M)$ and a one-form $\alpha\in\Omega^1(M)$ a $\binom11$-tensor field $\{X,\alpha\}=\boldsymbol{\mathsf{T}}(X,\alpha,\cdot,\cdot)$ which in particular can be viewed as a section of $\operatorname{End}(TM)$ or of $\operatorname{End}(T^*M)$. The explicit meaning of $\{X,\alpha\}$ depends on the structure in question. For the three examples we treat, explicit formulae for $ \{X,\alpha\}\in \End(TM)$ are in \eqref{fund-proj}, \eqref{fund-conf}, and \eqref{fund-Grass} below.

  Now given two vector fields $X,Y\in\mathfrak X(M)$ and a $T^*M$-valued one-form $\boldsymbol{\mathsf{P}}$, we define a $\binom13$-tensor field $\partial\boldsymbol{\mathsf{P}}$ by
\begin{equation}\label{partial-Rho}
  \partial\boldsymbol{\mathsf{P}}(X,Y):=\{X,\boldsymbol{\mathsf{P}}(Y)\}-\{Y,\boldsymbol{\mathsf{P}}(X)\}.
\end{equation}
By construction, this is skew symmetric in $X$ and $Y$ and thus defines a two-form with values in $\operatorname{End}(TM)$, so this looks like the curvature of a linear connection on $TM$. Now for a Weyl structure $\sigma$, we consider the induced Weyl connection $\theta_{\sigma}$. The curvature of this Weyl connection is equivalently encoded by a two-form $\mathbf{R}^\sigma$ with values in $\operatorname{End}(TM)$. Now it turns out that the normalization condition on the Cartan connection implies that the Rho tensor $\boldsymbol{\mathsf{P}}^\sigma$ is uniquely characterized by the fact that $\mathbf{R}^\sigma-\partial\boldsymbol{\mathsf{P}}^\sigma$ has vanishing Ricci type contraction, see \cite{MR1620484} or \cite[Section 5.2.3]{MR1997873}. Throughout the article we follow the convention of defining the Ricci curvature $\mathrm{Ric}(\nabla)$ of a torsion-free connection $\nabla$ as
\[
(X,Y) \mapsto \mathrm{Ric}(\nabla)(X,Y)=\tr\left (Z \mapsto R^{\nabla}(Z,X)(Y)
\right)
\]
where the curvature operator $R^{\nabla}$ is defined as usual by
\[
R^{\nabla}(X,Y)(Z)=\nabla_X\nabla_YZ-\nabla_Y\nabla_XZ-\nabla_{[X,Y]}Z. 
\]
Notice that this convention, while common in projective differential geometry, differs (by a swap of the arguments) from the standard convention in Riemannian geometry. 

\section{The almost para-K\"ahler structure of a Weyl structure}

\subsection{Construction of the almost para-K\"ahler structure}

We briefly review the construction of the almost para-K\"ahler structure associated to a torsion-free $|1|$-graded parabolic geometry given in \cite{MR4604420}. Notice that we may think of a torsion-free $|1|$-graded parabolic geometry $(\pi : \mathcal{G} \to M,\theta)$ of type $(G,P)$ on $M$ as a Cartan geometry $(\Pi : \mathcal{G} \to A,\theta)$ of type $(G,G_0)$ on the quotient $A:=\mathcal{G}/G_0$. In doing so, the tangent bundle of $A$ becomes $TA=\mathcal{G}\times_{G_0} (\mathfrak{g}/\mathfrak{g}_0)$, where $G_0$ acts via the adjoint action on $\mathfrak{g}$ and hence on $\mathfrak{g}/\mathfrak{g}_0$. The $G_0$-module $\mathfrak{g}/\mathfrak{g}_0$ is isomorphic to $\mathfrak{g}_{-1}\oplus \mathfrak{g}_1$, where again $G_0$ acts on both summands via the adjoint representation. Consequently, the tangent bundle of $A$ decomposes into a direct sum of rank $n$ vector bundles $TA=L^+\oplus L^{-}$, where $L^{\pm}=\mathcal{G}\times_{G_0} \mathfrak{g}_{\pm 1}$. We consider the $1$-form $\eta=\psi_B\circ \theta_1 \in \Omega^1(\mathcal{G},\mathfrak{g}_{-1}^*)$. Explicitly, we have
\[
\eta(v)(X)=B(\theta_1(v),X)
\]
for all $v \in T\mathcal{G}$ and all $X \in \mathfrak{g}_{-1}$. It follows from the properties of the Cartan connection and the invariance of $B$ under the adjoint representation that the $1$-form $\eta$ is $G_0$-equivariant and semibasic for the projection $\Pi : \mathcal{G} \to A$. Consequently, $\eta$ represents a $1$-form $\boldsymbol{\eta}$ on $A$ with values in $(L^{-})^*\subset T^*A$. Hence we may view $\boldsymbol{\eta}$ as a section of $T^*A\otimes T^*A$. The symmetric part $h=\operatorname{Sym}\boldsymbol{\eta}$ is then a pseudo-Riemannian metric of split-signature on $A$. This uses that $\psi_B : \mathfrak{g}_1\to\mathfrak{g}_{-1}^*$ is an isomorphism. The alternating part $\Omega=\operatorname{Alt} \boldsymbol{\eta}$ turns out to be a symplectic form by torsion-freeness (see \cite[Theorem 3.1]{MR4604420}), and the pair $(h,\Omega)$ is an almost para-K\"ahler structure. Furthermore, the sections of $A \to M$ are in bijective correspondence with the Weyl structures for $(\pi : \mathcal{G} \to M,\theta)$. Remarkably, by \cite[Theorem 3.5]{MR4604420}, the normalization conditions for $|1|$-graded parabolic geometries imply that the metric $h$ always is Einstein and hence $(h,\Omega)$ is an almost para-K\"ahler Einstein structure. We refer to \cite{MR4604420} for further details and to~\cite{MR4366946} for recent results about the geometry of $4$-dimensional para-K\"ahler Einstein structures.

\subsection{The choice of a Weyl structure}

In this section we shall prove the main structural identity \eqref{eqn:universalstructure}. We start by identifying $\mathcal{G}$ with the product $\mathcal{G}_0\times \mathfrak{g}_1$ equipped with a suitable right action. An element of $\mathcal{G}_0$ will be denoted by $[u]$, where $u \in \mathcal{G}$. On $\mathcal{G}_0\times \mathfrak{g}_1$ a right $P$-action $\hat{R}$ is defined by the rule
\[
([u],Z)\cdot g=([u\cdot g_0],\Ad(g_0^{-1})(Z)+W)
\]
for all $([u],Z) \in \mathcal{G}_0\times\mathfrak{g}_1$ and all $g=g_0\exp(W)\in P$ and as the basepoint projection we take the map
\[
\hat{\pi} : \mathcal{G}_0\times \mathfrak{g}_1 \to M, \quad ([u],Z) \mapsto \pi_0([u]). 
\] 
With these definitions, $\hat{\pi} : \mathcal{G}_0\times \mathfrak{g}_1 \to M$  is indeed a right principal $P$-bundle and moreover, the choice of a Weyl structure identifies this bundle with $\pi : \mathcal{G} \to M$:

\begin{prop}\label{prop:isomweylstructure}
Let $(\pi : \mathcal{G} \to M,\theta)$ be a $|1|$-graded parabolic geometry of type $(G,P)$. Then every Weyl structure $\sigma : \mathcal{G}_0 \to \mathcal{G}$ induces an isomorphism of principal $P$-bundles 
\[
\Phi_{\sigma} : \mathcal{G}_0\times \mathfrak{g}_1 \to \mathcal{G}, \qquad ([u],Z)\mapsto \sigma([u])\cdot \exp(Z)
\]
satisfying
\[
(\Phi_{\sigma})^*\theta=\d \mathcal{Z}+\sigma^*\theta+[\sigma^*\theta,\mathcal{Z}]+\frac{1}{2}[[\sigma^*\theta,\mathcal{Z}],\mathcal{Z}],
\]
where $\mathcal{Z} : \mathcal{G}_0 \times \mathfrak{g}_1 \to \mathfrak{g}_1$ denotes the projection onto the second factor, the brackets are in $\mathfrak g$, and we omit writing the pullbacks from the first factor. 
\end{prop}

For the proof we need the following elementary lemma on the Maurer-Cartan form $\Upsilon_P\in\Omega^1(P,\mathfrak p)$:
\begin{lem}\label{lem:pullbackexp}
The exponential map $\exp : \mathfrak{g}_1 \to P$ satisfies $\exp^*\Upsilon_P=\d\, \mathrm{Id}_{\mathfrak{g}_1}$, where $\mathrm{Id}_{\mathfrak{g}_1}$ denotes the identity map on $\mathfrak{g}_1$.  
\end{lem}
\begin{proof}
For $X,Y \in \mathfrak{g}_1$, we compute
\begin{align*}
(\exp^*\Upsilon_P)_X(Y)&=(\Upsilon_P)_{\exp(X)}(\exp^{\prime}_X(Y))=(L_{\exp(X)^{-1}})^{\prime}_{\exp(X)}(\exp^{\prime}_X(Y))\\
&=(L_{\exp(X)^{-1}}\circ \exp)^{\prime}_X(Y)\\
&=\left.\frac{\d}{\d t}\right|_{t=0}\exp(-X)\exp(X+tY)=\left.\frac{\d}{\d t}\right|_{t=0}\exp(tY)=Y,
\end{align*}
where we use that $[X,X+tY]=0$ since $[\mathfrak{g}_1,\mathfrak{g}_1]=\{0\}$. 
\end{proof}
\begin{proof}[Proof of \cref{prop:isomweylstructure}]
First observe that since $[\mathfrak{g}_1,\mathfrak{g}_1]=\{0\}$, we have 
\[
\exp(Z)\exp(W)=\exp(Z+W)
\] for all $Z,W \in \mathfrak{g}_1$. As a consequence, the standard identity 
\[
g\exp(X)g^{-1}=\exp(\Ad(g)(X)), \qquad g \in G, X \in \mathfrak{g}
\]
implies that for all $g_0\exp(W)\in P$ and $h_0\exp(Z) \in P$, we have 
\begin{equation}\label{eqn:lieid}
g_0\exp(Z)h_0\exp(W)=g_0h_0\exp(\Ad(h_0^{-1})(Z)+W),
\end{equation}
where we use that $\Ad(h_0^{-1})$ preserves $\mathfrak{g}_1$.

Since $\exp : \mathfrak{g}_1 \to P_+$ is a diffeomorphism and $\sigma : \mathcal{G}_0 \to \mathcal{G}$ is equivariant, it easily follows that $\Phi_{\sigma}$ is a diffeomorphism. In order to verify the equivariance of $\Phi_{\sigma}$, we compute for all $([u],Z) \in \mathcal{G}_0\times \mathfrak{g}_1$ and for all $g=g_0\exp(W) \in P$
\[
\begin{aligned}
\Phi_{\sigma}(([u],Z)\cdot g)&=\Phi_{\sigma}(([u\cdot g_0],\Ad(g_0^{-1})(Z)+W))\\
&=\sigma([u\cdot g_0])\cdot \exp(\Ad(g_0^{-1})(Z)+W)\\
&=\sigma([u])\cdot g_0\exp(\Ad(g_0^{-1})(Z)+W)\\
&=\sigma([u])\cdot\exp(Z)g_0\exp(W)\\
&=\sigma([u])\cdot \exp(Z)g=\Phi_{\sigma}(([u],Z))\cdot g,
\end{aligned}
\]
where we used the definitions of the various mappings as well as \eqref{eqn:lieid} and the equivariance of $\sigma$. It follows that $\Phi_{\sigma}$ is a principal $P$-bundle isomorphism.  

For the second part of the lemma we denote by $\Ad^{-1}$ the composition of $\Ad$ with the inversion in $P$ and compute 
\begin{equation}\label{eqn:pullbackomegapartial}
\begin{aligned}
(\Phi_{\sigma})^*\theta&=(\sigma,\exp)^*(R^*\theta)=(\sigma,\exp)^*\left(\Upsilon_P+\Ad^{-1}\circ\, \theta\right)\\
&=\exp^*\Upsilon_P+\sigma^*\left(\Ad^{-1}\circ\, \theta\right)
\end{aligned}
\end{equation}
where we used \eqref{eqn:pullbackrightaction} and think of $(\sigma,\exp)$ as a map $\mathcal{G}_0\times \mathfrak{g}_1 \to \mathcal{G}\times P$. Now for $Z,W \in \mathfrak{g}$ we have the standard identity
\[
\Ad(\exp(Z))(W)=\sum_{k=0}^\infty \frac{\ad(Z)^k}{k!}(W)=W+[Z,W]+\frac{1}{2}[Z,[Z,W]]+\cdots.
\]
As a consequence, we obtain for all $Z \in \mathfrak{g}_1$
\begin{equation}\label{eqn:pullbackrightaction1grad}
(R_{\exp(Z)})^*\theta=\Ad(\exp(-Z))\circ\, \theta=\theta+[\theta,Z]+\frac{1}{2}[[\theta,Z],Z],
\end{equation}
where we use that the sum terminates after three summands, since $[\mathfrak{g}_1,\mathfrak{g}_1]=0$. Combining \eqref{eqn:pullbackomegapartial}, \eqref{eqn:pullbackrightaction1grad} and \cref{lem:pullbackexp}, we obtain
\begin{equation}\label{eqn:pullbackcartanweylisom}
(\Phi_{\sigma})^*\theta=\d \mathcal{Z}+\sigma^*\theta+[\sigma^*\theta,\mathcal{Z}]+\frac{1}{2}[[\sigma^*\theta,\mathcal{Z}],\mathcal{Z}],
\end{equation}
as claimed. 
\end{proof}

Recall from \cref{subsec:weylstructures} that for every choice of Weyl structure $\sigma : \mathcal{G}_0 \to \mathcal{G}$, $\omega=\sigma^*\theta_{-1} \in \Omega^1(\mathcal{G}_0,\mathfrak{g}_{-1})$ is the soldering form of the $G_0$-structure $\pi_0 :\mathcal{G_0}\to M$. Moreover, $\vartheta_{\sigma}=\sigma^*\theta_0 \in \Omega^1(\mathcal{G}_0,\mathfrak{g}_0)$ is a principal connection on $\pi_0 : \mathcal{G}_0 \to M$. Using \eqref{eqn:pullbackalongweylstructure} and \eqref{eqn:pullbackcartanweylisom} we thus obtain
\[
(\Phi_{\sigma})^*\theta_1=\d\mathcal{Z}+\sigma^*\theta_1+[\vartheta_{\sigma},\mathcal{Z}]+\frac{1}{2}[[\omega,\mathcal{Z}],\mathcal{Z}]
\]
and hence, using \eqref{eqn:pullbackalongweylstructure} again, we have
\begin{equation}\label{eqn:yetanotherpullback}
(\Phi_{\sigma})^*(\psi_B\circ \theta_1)=\psi_B\circ\left(\d\mathcal{Z}+[\vartheta_{\sigma},\mathcal{Z}]\right)-\mathsf{P}^{\sigma}+\frac{1}{2}\psi_B\circ\left([[\omega,\mathcal{Z}],\mathcal{Z}]\right).
\end{equation}

In order to relate this to the Patterson--Walker metric associated to $\vartheta_{\sigma}$, we first observe that via $\psi_B$, the map $\mathcal Z$ corresponds to the second projection $\mathcal G_0\times\mathfrak g_{-1}^*\to \mathfrak g_{-1}^*$. Moreover, in the notation of \cref{sec:Patterson}, the expression $[\vartheta_{\sigma},\mathcal{Z}]$ can be written as $(\ad\circ\vartheta_{\sigma})(\mathcal Z)$. Invariance of the bilinear form $B$ implies that for $X \in \mathfrak{g}_{-1},Y \in \mathfrak{g}_0$ and $Z \in \mathfrak{g}_1$, we get
  $$
  B(\ad(Y)(X),Z)=-B(X,\ad(Y)(Z)). 
  $$
  This exactly says that, via $\psi_B$, the adjoint action on $\mathfrak g_1$ corresponds to the dual of the adjoint action on $\mathfrak g_{-1}$, which was denoted by $\varrho^*$ in \cref{sec:Patterson}. Together, this shows that the term $\psi_B\circ(\d\mathcal{Z}+[\vartheta_{\sigma},\mathcal{Z}])$ exactly gives the $\mathfrak g_{-1}^*$-valued $1$-form $\zeta_{\vartheta_\sigma}$ defined there. 

Combining this with \eqref{eqn:yetanotherpullback}, we obtain
\begin{equation}
(\Phi_{\sigma})^*(\psi_B\circ \theta_1)=\zeta_{\vartheta_\sigma}-\mathsf{P}^{\sigma}+q
\end{equation}
where $q \in \Omega^1(\mathcal{G}_0\times\mathfrak{g}_{-1}^*,\mathfrak{g}_{-1}^*)$ is given by
\begin{equation}\label{eqn:expressionforq}
q=\frac{1}{2}\psi_B\circ \left([[\omega,\psi_B^{-1}\circ \xi],\psi_B^{-1}\circ \xi]\right). 
\end{equation}
By construction, $q$ represents a $1$-form $\boldsymbol{q}$ on $T^*M$ with values in the pullback of the cotangent bundle of $M$, which is evidently closely related to the operation on $M$ introduced in \cref{sec:Schouten}. More precisely, for $\alpha\in T^*M$ with $\nu(\alpha)=x$ and a tangent vector $X\in T_\alpha T^*M$, we get
  \begin{equation}\label{q-formula}
    \boldsymbol{q}(\alpha)(X)=\tfrac12\{\nu^{\prime}_{\alpha}(X),\alpha\}(\alpha)\in T^*_xM=(\nu^*T^*M)_\alpha.
  \end{equation}
  In particular, this shows that $\boldsymbol{q}$ is semibasic for the projection $\nu : T^*M \to M$ and satisfies $(\mathcal{S}_{t})^*\boldsymbol{q}=t^2 \boldsymbol{q}$, where $\mathcal{S}_t : T^*M \to T^*M$ denotes scaling of a cotangent vector by the factor $t \in \R\setminus\{0\}$. 

Finally, notice that using $\psi_B$ to identify $\mathfrak g_1$ with $\mathfrak g_{-1}^*$, we get $T^*M\simeq \mathcal{G}_0 \times_{G_0} \mathfrak{g}_1$. But then the $G_0$-equivariant diffeomorphism $\Phi_\sigma: \mathcal{G}_0\times\mathfrak{g}_1 \to \mathcal{G}$ induces a diffeomorphism $\varphi_{\sigma} : T^*M \to \mathcal{G}/G_0=A$ and we have
\[
(\varphi_{\sigma})^*\boldsymbol{\eta}=\boldsymbol{\zeta}_{\vartheta_{\sigma}}-\nu^*\boldsymbol{\mathsf{P}}^{\sigma}+\boldsymbol{q}.
\]
Recall that $\boldsymbol{\eta}$ is the $(L^-)^*$-valued $1$-form on $A$ whose symmetric and alternating part give the almost para-K\"ahler structure $(h,\Omega)$ of the parabolic geometry $(\pi : \mathcal{G} \to M,\theta)$. Recall also that the symmetric part of the form $\boldsymbol{\zeta}_{\vartheta_{\sigma}}$ is the Patterson--Walker metric $h_{\vartheta_s}$ of the Weyl connection $\vartheta_{\sigma}$ determined by $\sigma$. Finally, viewed as a bilinear form on $T_\alpha T^*M$, $\boldsymbol{q}(\alpha)$ turns out to be symmetric. By construction, $\boldsymbol{q}(\alpha)$ is the pullback of a bilinear form on $T_xM$ with $x=\nu(\alpha)$. The latter is induced by the bilinear form on $\mathfrak g_{-1}$ that, for some fixed $Z\in\mathfrak g_1$, maps $(X_1,X_2)$ to
$$
\tfrac12B([[X_1,Z],Z],X_2)=-\tfrac12B([X_1,Z],[Z,X_2])=\tfrac12B([X_1,Z],[X_2,Z]),
$$
so this is obviously symmetric. In summary, we have thus shown:

\begin{thm}\label{thm:main}
Let $(\pi :\mathcal{G} \to M,\theta)$ be a torsion-free $|1|$-graded parabolic geometry with associated almost para-K\"ahler structure $(h,\Omega)$ on $A$ and $\sigma : \mathcal{G}_0 \to \mathcal{G}$ a choice of Weyl structure. Then we have
\[
(\varphi_\sigma)^*h=h_{\vartheta_\sigma}-\nu^*(\operatorname{Sym}\boldsymbol{\mathsf{P}}^{\sigma})+\boldsymbol{q}, 
\]
where $\boldsymbol{q}$ is given by formula \eqref{q-formula}.
\end{thm}

Of course, $(\varphi_\sigma)^*h$ is always an Einstein metric. Moreover,
  it is shown in \cite{MR4604420} that
  $$
(\varphi_\sigma)^*\Omega=-\d\tau+\nu^*(\operatorname{Alt}\boldsymbol{\mathsf{P}}^{\sigma}).
  $$ Hence \cref{thm:main} provides an analogue of \cref{prop:projective} and
  \cref{prop:conformal} for Weyl connections compatible with a torsion-free
  $|1|$-graded parabolic geometry. We will show in \cref{sec:examples} below
  how to specialize this to conformal and projective structures. 

\begin{rmk}[Local coordinate expression]
In terms of a choice of local coordinates $(x^i) : U \to \R^n$ on some open subset $U\subset M$, the metric $(\varphi_{\sigma})^*h$ takes the following explicit form. Let $(x^i,\xi_i) : \nu^{-1}(U) \to \R^{2n}$ denote the canonical coordinates induced on $\nu^{-1}(U)\subset T^*M$. The Weyl connection $\vartheta_{\sigma}$ induces a torsion-free connection $\nabla$ on $TM$ whose Christoffel symbols with respect to the coordinates $(x^i)$ we denote by $\Gamma^{i}_{jk}$. The Patterson--Walker metric of $\vartheta_{\sigma}$ can then be expressed as
\[
(\d\xi_i-\Gamma^k_{ij}\xi_k\d x^j)\odot \d x^i,
\] 
where $\odot$ denotes the symmetric tensor product. On $\nu^{-1}(U)$ we thus obtain
\[
(\varphi_\sigma)^*h=\left(\d\xi_i-\Gamma^k_{ij}\xi_k\d x^j-\mathsf{P}_{(ij)}+q_{ij}\right)\odot \d x^i,
\]
where we write $\boldsymbol{q}=q_{ij}\d x^i\otimes \d x^j$ for unique real-valued functions $q_{ij}=q_{ji} : U \to \R$ and $\boldsymbol{\mathsf{P}}^{\sigma}=\mathsf{P}_{ij}\d x^i\otimes \d x^j$ for unique real-valued functions $\mathsf{P}_{ij} : U \to \R$. Here and henceforth, we employ the summation convention and $\mathsf{P}_{(ij)}$ denotes symmetrization in the indices $i,j$. 
\end{rmk}

\section{Examples}\label{sec:examples}

\subsection{Projective geometry}\label{sec:projective}

Consider an $n$-dimensional manifold $M$ endowed with a projective structure $[\nabla]$, a class of torsion-free connections that have the same geodesics up to parametrization. This determines a $|1|$-graded parabolic geometry $(\pi : \mathcal{G} \to M,\theta)$, where $G=\mathrm{SL}_{\pm}(n+1,\R)$ is the subgroup of $\mathrm{GL}(n+1,\R)$ consisting of matrices whose determinant is $\pm 1$. The grading of its Lie algebra $\mathfrak{g}=\mathfrak{sl}(n+1,\R)=\{B \in \mathfrak{gl}(n+1,\R)\,|\, \tr B=0\}$ is given by
\[
\mathfrak{g}_{-1}=\left\{\left.\begin{pmatrix} 0 & 0 \\ x & 0 \end{pmatrix} \right | x \in \R^n\right\}, \quad \mathfrak{g}_{1}=\left\{\left.\begin{pmatrix} 0 & y \\ 0 & 0 \end{pmatrix} \right | y \in \R^{n*}\right\}
\]
and
\[
\mathfrak{g}_{0}=\left\{\left.\begin{pmatrix} -\tr A & 0 \\ 0 & A \end{pmatrix} \right | A \in \mathfrak{gl}(n,\R)\right\}.
\]
Here $B$ is normalised so that
\[
B\left(\begin{pmatrix} 0 & 0 \\ x & 0 \end{pmatrix},\begin{pmatrix} 0 & y \\ 0 & 0 \end{pmatrix}\right)=yx=y(x).
\]

Next, one computes that for $x\in\mathbb R^n$ and $y,z\in\mathbb R^{n*}$ we get
\begin{equation}\label{bracket-formula}
\left[\left[\begin{pmatrix} 0 & 0 \\ x & 0 \end{pmatrix},\begin{pmatrix} 0 & y \\ 0 & 0 \end{pmatrix}\right],\begin{pmatrix} 0 & z \\ 0 & 0 \end{pmatrix}\right]=\begin{pmatrix} 0 & -(yx)z-(zx)y \\ 0 & 0 \end{pmatrix}.
\end{equation}
This shows that for $\xi\in T_xM$ and $\alpha\in T^*_xM$ we get $\{\xi,\alpha\}(\alpha)=-2\alpha(\xi)\alpha$. Together with formula \eqref{q-formula} and the definition of the tautological form $\tau$ on $T^*M$, this shows that
\[
\boldsymbol{q}=-\tau\otimes \tau,
\] 
in agreement with \cite{MR3833818}. 

The Weyl connection $\vartheta_{\sigma}=\sigma^*\theta_0$ determines a torsion-free connection $\nabla$ on $TM$ which is a representative connection of the projective structure. To compute the associated Rho tensor a similar computation as for formula \eqref{bracket-formula} shows that for $X,Z\in\mathfrak X(M)$ and $\alpha\in\Omega^1(M)$, we get
\begin{equation}\label{fund-proj}
  \{X,\alpha\}(Z)=\alpha(X)Z+\alpha(Z)X.
\end{equation}
Using this and formula \eqref{partial-Rho} from \cref{sec:Schouten} we conclude that for $X,Y,Z\in\mathfrak X(M)$, we obtain
$$
\partial\boldsymbol{\mathsf{P}}^\sigma(X,Y)(Z)=\boldsymbol{\mathsf{P}}^\sigma(Y,X)Z-\boldsymbol{\mathsf{P}}^\sigma(X,Y)Z+\boldsymbol{\mathsf{P}}^\sigma(Y,Z)X-\boldsymbol{\mathsf{P}}^\sigma(X,Z)Y. 
$$
This easily implies that the Ricci type contraction of $\partial\boldsymbol{\mathsf{P}}^\sigma$ maps $X,Y$ to $n\boldsymbol{\mathsf{P}}^\sigma(X,Y)-\boldsymbol{\mathsf{P}}^\sigma(Y,X)$. Now let $\operatorname{Ric}(\nabla)$ be the Ricci type contraction of the curvature of $\vartheta_{\sigma}$. Then the discussion in \cref{sec:Schouten} shows that we must have
$$
\operatorname{Ric}(\nabla)(X,Y)=n\boldsymbol{\mathsf{P}}^\sigma(X,Y)-\boldsymbol{\mathsf{P}}^\sigma(Y,X). 
$$
Symmetrizing and alternating, we deduce that $\operatorname{Sym}\operatorname{Ric}(\nabla)=(n-1)\operatorname{Sym}\boldsymbol{\mathsf{P}}^\sigma$ and $\operatorname{Alt}\operatorname{Ric}(\nabla)=(n+1)\operatorname{Alt}\boldsymbol{\mathsf{P}}^\sigma$, and hence
\[
\boldsymbol{\mathsf{P}}^{\sigma}=\frac{1}{(n-1)}\operatorname{Sym}\operatorname{Ric}(\nabla)+\frac{1}{(n+1)}\operatorname{Alt}\operatorname{Ric}(\nabla).
\]

Hence we see that for projective structures our main \cref{thm:main} implies the \cref{prop:projective}. 

\begin{rmk}[Dancing metric]
Starting from the standard projective structure on $\mathbb{RP}^2$, the resulting para-K\"ahler--Einstein structure is defined on $A=\mathrm{SL}(3,\R)/\mathrm{GL}(2,\R)$. In this case the Einstein metric is referred to as the \emph{dancing metric} \cite{MR3811534,MR4401805} because of its significance in the ``rolling'' of the projective planes $\mathbb{RP}^2$ and $\mathbb{RP}^{2*}$.  This para-K\"ahler--Einstein structure was first constructed in \cite{MR0066020} (in any dimension).
\end{rmk}

\begin{rmk}[para-c-projective compactification]
In the projective case, the almost para-K\"ahler structure on $A$ admits a so-called para-c-projective compactification, see \cite{MR4078278}, an analogue of a c-projective compactification, see \cite{MR3956522}.
\end{rmk}

\subsection{Conformal geometry}\label{sec:conformal}
A conformal manifold $(M,[g])$ of dimension $n\geqslant 3$ gives rise to a $|1|$-graded parabolic geometry $(\pi : \mathcal{G} \to M,\theta)$ where $G$ is defined as follows: Consider the matrix
\[
J=\begin{pmatrix} 0 & 0 & -1 \\ 0 & \mathrm{I}_n & 0 \\ -1 & 0 & 0 \end{pmatrix}
\]
of size $n+2$ and let $G=\mathrm{O}(n+1,1)$ denote the subgroup of $\mathrm{GL}(n+2,\R)$ consisting of matrices $a$ satisfying $a^tJa=J$. The Lie algebra $\mathfrak{g}=\mathfrak{o}(n+1,1)$ of $G$ consists of matrices of the form 
\[
\begin{pmatrix}
s & z & 0  \\ x & A & z^t \\ 0 & x^t & -s
\end{pmatrix}
\]
where $s \in \R$, $x\in \R^n$, $z\in\R^{n*}$ and $A \in \mathfrak{o}(n)$ is a skew-symmetric matrix of size $n$. The grading of $\mathfrak{g}$ is given by
\[
\mathfrak{g}_{-1}=\left\{\left.\begin{pmatrix} 0 & 0 & 0 \\ x & 0 & 0 \\ 0 & x^t & 0 \end{pmatrix}\right| x \in \R^n\right\}, \quad \mathfrak{g}_{1}=\left\{\left.\begin{pmatrix} 0 & z & 0 \\ 0 & 0 & z^t \\ 0 & 0 & 0 \end{pmatrix}\right| z \in \R^{n*}\right\}
\]
and
\[
\mathfrak{g}_{0}=\left\{\left.\begin{pmatrix} s & 0 & 0 \\ 0 & A & 0 \\ 0 & 0 & -s \end{pmatrix}\right| s \in \R, A \in \mathfrak{o}(n)\right\}.
\]
We normalise $B$ such that
\[
B\left(\begin{pmatrix} 0 & 0 & 0 \\ x & 0 & 0 \\ 0 & x^t & 0 \end{pmatrix},\begin{pmatrix} 0 & z & 0 \\ 0 & 0 & z^t \\ 0 & 0 & 0 \end{pmatrix}\right)=z(x)=zx.
\]

Computing triple brackets as for formula \eqref{bracket-formula} one verifies that (with obvious notation) we get for $x,y\in\R^n$ and $z,w\in\R^{n*}$ the expressions
\begin{gather}\label{bracket-conf-1}
   [[x,z],w]=-z(x)w-w(x)z+(w^t\cdot z^t)x^t\\ \label{bracket-conf-2}  [[x,z],y]=z(x)y+z(y)x-(x\cdot y)z^t
\end{gather}

Now formula \eqref{bracket-conf-1} shows that for $\xi\in T_xM$ and $\alpha\in T^*_xM$ we get
$$
\{\xi,\alpha\}(\alpha)=2\alpha(\xi)\alpha+g^\#_x(\alpha,\alpha)g_x(\xi,\cdot).
$$
Here $g$ is some metric from the conformal class and $g^\#$ is its dual metric, which immediately implies that the operation is conformally invariant. Together with formula \eqref{q-formula} and the definition of the tautological form $\tau$ on $T^*M$, this shows that 
\[
\boldsymbol{q}=-\tau \otimes \tau+\frac{1}{2}|\cdot|^2_{g^{\sharp}}\nu^*g. 
\]
Here $|\cdot|^2_{g^{\sharp}}$ is interpreted as a real-valued smooth function on $T^*M$ and the pullback $\nu^*g$ is interpreted as a one-form on $T^*M$ with values in $\nu^*T^*M$. 

The Weyl connection $\vartheta_{\sigma}=\sigma^*\theta_0$ determines a torsion-free connection $\nabla$ on $TM$ which preserves $[g]$ in the sense that for some (any hence any) representative metric $g \in [g]$ there exists a $1$-form $\beta$ such that
\[
\nabla g=\beta\otimes g. 
\]
 
    To compute the Rho-tensor, we first conclude from formula \eqref{bracket-conf-2} that for vector fields $X,Y\in\mathfrak X(M)$ and a one-form $\alpha\in\Omega^1(M)$, we get
    \begin{equation}\label{fund-conf}
    \{X,\alpha\}(Y)=\alpha(X)Y+\alpha(Y)X-g(X,Y)g^\#(\alpha,\cdot). 
    \end{equation}
    Using this and formula \eqref{partial-Rho} from \cref{sec:Schouten} we conclude that for $X,Y,Z\in\mathfrak X(M)$, we can write $\partial\boldsymbol{\mathsf{P}}^\sigma(X,Y)(Z)$ as
    \begin{align*}
      &\boldsymbol{\mathsf{P}}^\sigma(Y,X)Z+\boldsymbol{\mathsf{P}}^\sigma(Y,Z)X-g(X,Z)g^\#(\boldsymbol{\mathsf{P}}^\sigma(Y),\cdot)\\
      -&\boldsymbol{\mathsf{P}}^\sigma(X,Y)Z-\boldsymbol{\mathsf{P}}^\sigma(X,Z)Y+g(Y,Z)g^\#(\boldsymbol{\mathsf{P}}^\sigma(X),\cdot). 
    \end{align*}
    To form the Ricci-type contraction of this, we have to take a local orthonormal frame, insert each element for $X$ and then take the inner product with the same element and sum the results. This sends $(Y,Z)$ to
    $$
(n-1)\boldsymbol{\mathsf{P}}^\sigma(Y,Z)-\boldsymbol{\mathsf{P}}^\sigma(Z,Y)+g(Y,Z)\tr_{g^\#}(\boldsymbol{\mathsf{P}}^\sigma). 
    $$
    Observe that in the literature on conformal geometry usually only the case of Levi-Civita connections of metrics in the conformal class is discussed, for which $\boldsymbol{\mathsf{P}}^\sigma$ is automatically symmetric. Anyway, the discussion in \cref{sec:Schouten} shows that the above expression has to coincide with $\operatorname{Ric}(\nabla)$. To conclude the discussion as in \cref{sec:projective} above, we now have to compute the alternation, and the trace-free part and the trace part of the symmetrization, which gives
    \begin{align*}
      &\operatorname{Sym}_0\operatorname{Ric}(\nabla)=(n-2)\operatorname{Sym}_0\boldsymbol{\mathsf{P}}^\sigma \\
      &\operatorname{Alt}\operatorname{Ric}(\nabla)=n\operatorname{Alt}\boldsymbol{\mathsf{P}}^\sigma \\
      &g\tr_{g^\#}(\operatorname{Ric}(\nabla))= (2n-2)g\tr_{g^\#}(\boldsymbol{\mathsf{P}}^\sigma).
    \end{align*}
    Observe that for a Levi-Civita connection $\tr_{g^\#}(\operatorname{Ric}(\nabla))$ is the scalar curvature of the metric $g$. In any case, we immediately get the general formula 
     
\[
\boldsymbol{\mathsf{P}}^{\sigma}=\frac1{(n-2)}\operatorname{Sym}_0\operatorname{Ric}(\nabla)+\frac1n\operatorname{Alt}\operatorname{Ric}(\nabla)+\frac{1}{n(n-2)}g\tr_{g^\#}(\operatorname{Ric}(\nabla)).
\]

This completes the proof of \cref{prop:conformal} in the
  Riemannian case. For the pseudo-Riemannian case, the proof is completely
  analogous. 
  
\subsection{Grassmannian geometry}\label{sec:Grassmannian}

An \emph{almost Grassmannian structure of type $(m,n)$} on a manifold $M$ consists of two real vector bundles $E$ and $F$ on $M$, of rank $m$ and $n$ and vector bundle isomorphisms $TM \simeq E^*\otimes F\simeq \operatorname{Hom}(E,F)$ and $\Lambda^mE^*\cong\Lambda^nF$. Here $E^*$ denotes the dual of $E$ and the isomorphism between the top exterior powers will not be relevant for us. An almost Grassmannian structure on $M$ gives rise to a $|1|$-graded parabolic geometry $(\pi : \mathcal{G} \to M,\theta)$ where $G=\mathrm{SL}(n+m,\R)$. The structure is called \textit{Grassmannian} if it admits a compatible torsion-free connection on $TM$, which is equivalent to torsion-freeness of the parabolic geometry. The grading of the Lie algebra $\mathfrak{g}=\mathfrak{sl}(n+m,\R)$ of $G$ is given by
\[
\mathfrak{g}_{-1}=\left\{\left.\begin{pmatrix} 0 & 0 \\ x & 0 \end{pmatrix}\right| x \in M(n\times m,\R)\right\}, \quad \mathfrak{g}_{1}=\left\{\left.\begin{pmatrix} 0 & z \\ 0 & 0 \end{pmatrix}\right| z \in M(m\times n,\R)\right\}
\]
and
\[
\mathfrak{g}_{0}=\left\{\left.\begin{pmatrix} B & 0 \\ 0 & A \end{pmatrix}\right| A \in \mathfrak{gl}(n,\R), B \in \mathfrak{gl}(m,\R), \tr(A)+\tr(B)=0\right\},
\]
where $M(n\times m,\R)$ denotes the vector space of $(n\times m)$-matrices with real entries. The main case of interest for our purpose is $m=2$, $n\geq 2$, for which there are examples of such geometries that are torsion-free but not locally isomorphic to $G/P$. For $m=n=2$ such a structure is equivalent to a conformal structure of neutral signature. For $m,n>2$, any torsion-free structure is locally isomorphic to $G/P$, but our results still are of interest, since there is the freedom in the choice of Weyl structure.

  As the invariant form $B$ we use the trace form, which leads to $$B\left(\begin{pmatrix} 0 & 0 \\ x & 0 \end{pmatrix},\begin{pmatrix} 0 & z \\ 0 & 0 \end{pmatrix}\right)=\tr(zx)=\tr(xz).$$ Formally, the setup looks very similar to projective structures (which correspond to the case $m=1$). This is also reflected in the structure of the triple brackets, which formally look very similar to \eqref{bracket-formula}: For $x,y\in M(n\times m,\R)$ and $z,w\in M(m\times n,\R)$ we get (in obvious notation)
  \begin{equation}\label{Grass-brackets}
    [[x,z],w]=-zxw-wxz \qquad [[x,z],y]=xzy+yzx.
  \end{equation}
  However, here we have matrix multiplications so for example $xz$ is a $2\times 2$-matrix and $zxw$ is not simply a multiple of $w$.

  The easiest way to encode our operations geometrically is to define one additional operation on a manifold $M$ endowed with an almost Grassmannian structure. Since $TM\cong E^*\otimes F$, we get $T^*M\cong E\otimes F^*\cong \operatorname{Hom}(F,E)$ and thus composition of linear maps induces a bilinear bundle map $T^*M\times TM\to \operatorname{End}(E)$, which we denote by $(\alpha,X)\mapsto \alpha\circ X$ both on elements and on sections. This can be viewed as an ``refinement'' of the dual pairing, since by definition we get $\alpha (X)=\tr(\alpha\circ X)$, where $\tr$ denotes the point-wise trace. There clearly are analogous composition operations $TM\times \operatorname{End}(E)\to TM$ and $\operatorname{End}(E)\times T^*M\to T^*M$ (and others that we don't need here). In this language, the first formula in \eqref{Grass-brackets} readily shows that for $\alpha\in T^*_xM$ and $\xi\in T_xM$, we get $\{\xi,\alpha\}(\alpha)=-2(\alpha\circ\xi)\circ\alpha$. 

  Next, we get a corresponding refinement $\tau^G\in\Omega^1(T^*M,\operatorname{End}(\nu^*E))$ of the tautological one-form $\tau$ on $T^*M$. By definition, for $\alpha\in T^*M$ with $\nu(\alpha)=x\in M$, the fibre of $\operatorname{End}(\nu^*E)$ over $\alpha$ equals $\operatorname{End}(E_x)$, so for $\xi\in T_\alpha T^*M$, we can define
  $$
    \tau^G(\alpha)(\xi):=\alpha\circ \nu^{\prime}_{\alpha}(\xi). 
    $$
    By definition, the tautological form $\tau$ is then recovered as $\tau=\tr(\tau^G)$, where again $\tr$ denotes a point-wise trace in the values of the form. Using formula \eqref{q-formula} we readily conclude that for $\alpha\in T^*M$ and $\xi,\eta\in T_\alpha T^*M$ we get 
\begin{equation}\label{eqn:q-Grass}
\boldsymbol{q}(\alpha)(\xi,\eta)=-\tr((\alpha\circ \nu^{\prime}_{\alpha}(\xi))\circ(\alpha\circ \nu^{\prime}_{\alpha}(\eta))), 
\end{equation}
    which is evidently symmetric in $\xi$ and $\eta$. To connect more closely to the other cases, we can write this as $\boldsymbol{q}=-\tr(\tau^G\otimes\tau^G)$, where we agree that the tensor product of one-forms with values in $\operatorname{End}(\nu^*E)$ includes a composition of the values, i.e.\ $(A\otimes B)(\xi,\eta)=A(\xi)\circ B(\eta)$. 

    The description of the Rho tensor is also similar to the projective case, with some complications caused by the matrix multiplications.The Weyl connection $\vartheta_{\sigma}=\sigma^*\theta_0$ determines a torsion-free connection $\nabla$ on $TM$ which is induced by connections on $E$ and $F$ that are compatible with the isomorphism of the top exterior powers. Now the second equation in \eqref{Grass-brackets} readily shows that for $X,Y\in\mathfrak X(M)$ and $\alpha\in\Omega^1(M)$ we get
    \begin{equation}\label{fund-Grass}
      \{X,\alpha\}(Y)=-X\circ (\alpha\circ Y)-Y\circ (\alpha\circ X).
    \end{equation}
    Using this and formula \eqref{partial-Rho} from \cref{sec:Schouten} we conclude that for $X,Y,Z\in\mathfrak X(M)$, $\partial\boldsymbol{\mathsf{P}}^\sigma(X,Y)(Z)$ is given by
$$
X\circ (\boldsymbol{\mathsf{P}}^\sigma(Y)\circ Z)+Z\circ  (\boldsymbol{\mathsf{P}}^\sigma(Y)\circ X) - Y\circ (\boldsymbol{\mathsf{P}}^\sigma(X)\circ Z)-Z\circ  (\boldsymbol{\mathsf{P}}^\sigma(X)\circ Y).
$$
To compute the action of the Ricci type contraction on $Y$ and $Z$, we have to insert the elements of a basis for $X$ and contract (i.e.\ take the trace of the composition) with the dual basis element and sum up over the basis. To write up the result, we need additional notation. We have to view $\boldsymbol{\mathsf{P}}^\sigma(x)$ as an element in $\otimes^2 T^*_xM$ and identifying $T_x^*M$ with $\operatorname{Hom}(F_x,E_x)$ such an element defines a linear map $F_x\otimes F_x\to E_x\otimes E_x$. But for such a map, we can separately swap the input or the output and hence independently symmetrize or alternate in the input and the output. Writing $t_E$ for the twist map on $E\otimes E$ and $t_F$ for the twist map on $F\otimes F$, one verifies that the Ricci type contraction of $\partial\boldsymbol{\mathsf{P}}^\sigma$ can be written as 
$$
(mn+1)\boldsymbol{\mathsf{P}}^\sigma-t_E\circ \boldsymbol{\mathsf{P}}^\sigma - \boldsymbol{\mathsf{P}}^\sigma\circ t_F. 
$$
Observe that this is consistent with the result in the projective case, since for $m=1$, we have $t_E=\text{id}$. By the discussion in \cref{sec:Schouten}, this again has to coincide with the Ricci type contraction $\operatorname{Ric}(\nabla)$ of the curvature of $\nabla$. Now we can compose this equation on both sides with either a symmetrization or an alternation to obtain
\begin{gather*}
  \operatorname{Sym}\circ \operatorname{Ric}(\nabla)\circ \operatorname{Sym}=(mn-1)\operatorname{Sym}\circ \boldsymbol{\mathsf{P}}^\sigma\circ \operatorname{Sym} \\
  \operatorname{Alt}\circ \operatorname{Ric}(\nabla)\circ \operatorname{Alt}=(mn+3)\operatorname{Alt}\circ \boldsymbol{\mathsf{P}}^\sigma\circ \operatorname{Alt}\\
   \operatorname{Sym}\circ \operatorname{Ric}(\nabla)\circ \operatorname{Alt}=(mn+1)\operatorname{Sym}\circ \boldsymbol{\mathsf{P}}^\sigma\circ \operatorname{Alt} \\
  \operatorname{Alt}\circ \operatorname{Ric}(\nabla)\circ \operatorname{Sym}=(mn+1)\operatorname{Alt}\circ \boldsymbol{\mathsf{P}}^\sigma\circ \operatorname{Sym}.
\end{gather*}
From this, one deduces an explicit formula for $\boldsymbol{\mathsf{P}}^\sigma$ as
before and together with formula \eqref{eqn:q-Grass} this provides the
  explicit specialization of \cref{thm:main} to Weyl connections for a
  Grassmannian structure.

\providecommand{\mr}[1]{\href{http://www.ams.org/mathscinet-getitem?mr=#1}{MR~#1}}
\providecommand{\zbl}[1]{\href{http://www.zentralblatt-math.org/zmath/en/search/?q=an:#1}{zbM~#1}}
\providecommand{\arxiv}[1]{\href{http://www.arxiv.org/abs/#1}{arXiv:#1 }}
\providecommand{\doi}[1]{\href{http://dx.doi.org/#1}{DOI~#1}}
\providecommand{\href}[2]{#2}

\end{document}